\documentclass[11pt]{amsproc}
\author{Komal Agrawal}
\author{Paul Pollack}
\address{Department of Mathematics\\ University of Georgia\\ Athens, GA 30602} 

\email{kpa43240@uga.edu}
\email{pollack@uga.edu}

\title{Finite sets containing near-primitive roots}
\usepackage{amsmath,amssymb,amsthm,microtype,geometry,mathscinet}
\usepackage[utf8]{inputenc}

\DeclareMathAlphabet{\curly}{U}{rsfs}{m}{n}
\newtheorem{thm}{Theorem}[section]
\newtheorem{cor}[thm]{Corollary}
\newtheorem{prop}[thm]{Proposition}
\newtheorem{lem}[thm]{Lemma}

\theoremstyle{remark}
\newtheorem*{rmk}{Remark}
\newtheorem{manualtheoreminner}{Theorem}
\newenvironment{manualtheorem}[1]{%
  \renewcommand\themanualtheoreminner{#1}%
  \manualtheoreminner
}{\endmanualtheoreminner}

\begin{document}
\renewcommand{\labelenumi}{(\roman{enumi})}
\def\Ll{\mathcal{L}}
\def\N{\mathbb{N}}
\def\Aa{\mathcal{A}}
\def\Q{\mathbb{Q}}
\newcommand\rad{\mathrm{rad}}
\def\Z{\mathbb{Z}}
\def\F{\mathbb{F}}
\def\R{\mathbb{R}}
\def\C{\mathbb{C}}
\def\Pp{\mathcal{P}}
\newcommand\Li{\mathrm{Li}}
\begin{abstract} Fix $a \in \Z$, $a\notin \{0,\pm 1\}$. A simple argument shows that for each $\epsilon > 0$, and almost all (asymptotically 100\% of) primes $p$, the multiplicative order of $a$ modulo $p$ exceeds $p^{\frac12-\epsilon}$. It is an open problem to show the same result with $\frac12$ replaced by any larger constant. We show that if $a,b$ are multiplicatively independent, then for almost all primes $p$, one of $a,b,ab, a^2b, ab^2$ has order exceeding $p^{\frac{1}{2}+\frac{1}{30}}$. The same method allows one to produce, for each $\epsilon > 0$, explicit finite sets $\Aa$ with the property that for almost all primes $p$, some element of $\Aa$ has order exceeding $p^{1-\epsilon}$. Similar results hold for orders modulo general integers $n$ rather than primes $p$.
\end{abstract}

\maketitle

\section{Introduction}
Let $\ell_a(p)$ denote the multiplicative order of the integer $a$ modulo the prime number $p$. A celebrated 1927 conjecture of Artin (\textsf{Artin's primitive root conjecture})  asserts that if $a \in \Z$, with $a$ not a square and $a\ne -1$, then there are infinitely many primes $p$ for which $\ell_a(p) = p-1$. Artin's conjecture remains unresolved, but there has been significant progress. Hooley showed in 1967 that the conjecture is implied by the Riemann Hypothesis for a certain class of Dedekind zeta functions \cite{hooley67}, while Heath-Brown \cite{HB86} (building on earlier work of Gupta and Murty \cite{GM84}) showed in 1986 that Artin's conjecture holds for all prime values of $a$ with at most two exceptions.

Hooley's ideas in \cite{hooley67} have several other nice consequences for the distribution of the numbers $\ell_a(p)$. Of interest to us is the following example: Assume the same Generalized Riemann Hypothesis as in \cite{hooley67}, and fix an integer $a\notin \{0,\pm 1\}$. If $\psi(x)$ is any function of $x$ tending to infinity as $x\to\infty$, then $\ell_a(p) > p/\psi(p)$ for almost all primes $p\le x$, meaning all  $p\le x$ with at most $o(x/\log{x})$ exceptions. (See \cite[Theorem 23]{KP05} for a quantitatively precise statement along these lines.) Thus, loosely speaking, $a$ is `nearly' a primitive root mod $p$ for almost all primes $p$.

The known unconditional results in this direction are considerably weaker. The following easy observation is implicit in \cite{hooley67} (see p.\ 212 there). Fix an integer $a\notin \{0,\pm 1\}$. Then for each $y\ge 1$, the product $\prod_{n \le y} (a^n-1)$ has absolute value $\exp(O(y^2))$, and so has $O(y^2)$ distinct prime factors. But that product is divisible by all primes $p$ with $\ell_a(p) \le y$. Taking $y=x^{1/2}/\log{x}$, we deduce that all but $o(x/\log{x})$ primes $p\le x$  have $$\ell_a(p) > x^{1/2}/\log{x} > p^{1/2}/\log{p}.$$

In \cite{erdos76}, Erd\H{o}s proves by a substantially more intricate argument that $\ell_a(p) > p^{1/2}$ for almost all primes $p$. At the end of the same article, he claims that if $\epsilon(x)$ is any positive-valued function tending to $0$ as $x\to\infty$, then \begin{equation}\label{eq:epsilonimprovement} \ell_a(p) > p^{\frac{1}{2}+\epsilon(p)} \end{equation} for almost all primes $p$. The proof appeared 20 years later in a joint a paper with Ram Murty \cite{EM99}.  It seems that \eqref{eq:epsilonimprovement} still holds the record as far as a lower bound for $\ell_a(p)$ that holds almost always, and that a new idea will be required to replace  $\frac{1}{2}$ with $\frac{1}{2}+\delta$ for some $\delta\gg 1$. The purpose of this article is to observe that, by considering simultaneously multiple values of $a$, one can overcome this barrier.

In our first theorem, we produce sets of $5$ integers, at least one of which almost always has order
$$ \gtrapprox p^{\frac{8}{15}} = p^{\frac{1}{2} + \frac{1}{30}}. $$ Recall that $a_1,\dots,a_k \in \Q^{\times}$ are called \textsf{multiplicatively independent} if whenever $a_1^{e_1} \cdots a_k^{e_k}=1$ with integers $e_1,\dots,e_k$, we have $e_1=e_2=\dots=e_k=0$.

\begin{thm}\label{thm:basic} Let $a$ and $b$ be nonzero integers that are multiplicatively independent. For almost all primes $p$, at least one of $a$, $b$, $ab$, $a^2 b$, and $ab^2$ has multiplicative order exceeding
\begin{equation}\label{eq:basically815} p^{8/15}/\exp(2\sqrt{\log{p}}). \end{equation}
\end{thm}

It would be easy to prove Theorem \ref{thm:basic} with a somewhat smaller quantity in the denominator of \eqref{eq:basically815}, but this is of no importance. Fix $\delta \in (0,1)$. Then as $\epsilon\to 0$, the lower density of primes $p$ (relative to the full set of primes) for which
\begin{equation*}
\text{$p-1$ has no divisor in the interval $[p^{\delta-\epsilon}, p^{\delta+\epsilon}]$}
\end{equation*}
tends to $1$. (This is essentially Erd\H{o}s and Murty's Theorem 2 in \cite{EM99}. It also follows from Proposition \ref{prop:ford} below.) From this, we deduce immediately that Theorem \ref{thm:basic}, with any function of size $p^{o(1)}$ in the denominator of \eqref{eq:basically815}, can be bootstrapped to yield the following result.

\begin{manualtheorem}{1$'$}\label{thm:basic2} Let $\epsilon(x)$ be a positive-valued function of $x\ge 2$ that tends to $0$ as $x\to\infty$. Let $a$ and $b$ be nonzero integers that are multiplicatively independent. Then for almost all primes $p$, at least one of $a$, $b$, $ab$, $a^2 b$, and $ab^2$ has multiplicative order exceeding
\begin{equation*} p^{\frac{8}{15}+\epsilon(p)}. \end{equation*}
\end{manualtheorem}

Matthews proved in 1982 that if $a$ and $b$ are multiplicatively independent integers, then the subgroup of ${\F_p}^{\times}$ generated by $a$ and $b$ has size at least $p^{2/3}/\log{p}$ for almost all primes $p$. Of course, $\frac{2}{3} > \frac{8}{15}$. What is novel about Theorem \ref{thm:basic} (and Theorem \ref{thm:basic2}) is that we can pinpoint an \emph{explicit, finite} subset of  $\langle a,b\rangle$ one of whose elements almost always generates a subgroup of size substantially larger than $p^{1/2}$.

In fact, Matthews's result implies that if $a_1,\dots,a_k$ is any finite list of multiplicatively independent integers, then $\langle a_1,\dots,a_k\rangle$  almost always has size at least $p^{1-\frac{1}{k+1}}/\log{p}$. We prove the following.

\begin{thm}\label{thm:general} Let $a_1,\dots,a_k$ be nonzero, multiplicatively independent integers. Let $N$ be a positive integer, and let
\[ \Aa = \{a_1^{e_1} a_2^{e_2} \cdots a_k^{e_k}: \text{each }0 \le e_i < N, \text{not all $e_i=0$}\}. \]
For almost all primes $p$, there is an $a\in \Aa$ with
\[ \ell_a(p) > p^{\delta}, \quad\text{where}\quad \delta = \left(1-\frac{1}{k+1}\right)\left(1-\frac{1}{N}\right). \]
\end{thm}

\noindent We highlight two immediate consequences of Theorem \ref{thm:general}:
\begin{enumerate}\item[(a)] For each $\epsilon \in (0,1)$, there are finite sets $\Aa$ of size $\exp(O(\frac{1}{\epsilon}\log\frac{1}{\epsilon}))$ such that, for almost all primes $p$, some $a \in \Aa$ has order exceeding $p^{1-\epsilon}$.
\item[(b)] Let $\xi(x)$ be any function that tends to infinity as $x\to\infty$. For each fixed $\epsilon > 0$, and almost all primes $p$, we have $L(p) > p^{1-\epsilon}$, where $L(p)$ is the maximum order mod $p$ of any of $1,2,3,\dots,\lfloor \xi(p)\rfloor$. 
\end{enumerate}

Ska\l{}ba has conjectured that almost all primes $p$ have a positive multiple of the form $2^m+2^n+1$, with positive integers $m,n$. He proved this assuming $\ell_2(p) > p^{0.8}$ for almost all primes $p$ \cite{skalba04}. (Thus the conjecture becomes a theorem if we assume GRH.) Elsholtz has shown, unconditionally, that Ska\l{b}a's  conjecture holds with $2^m + 2^n+1$ replaced by $2^{m_1} + 2^{m_2} + \dots + 2^{m_6} + 1$ \cite{elsholtz16}. Our proof of Theorem \ref{thm:general} yields another unconditional variant of Ska\l{b}a's conjecture.

\begin{cor}\label{cor:skalbavariant} Let $A = (2\cdot 3\cdot 5\cdot 7\cdot 11)^{9}$. For almost all primes $p$, there are positive integers $m,n$, such that $p$ divides $\prod_{a \mid A,\,a>1} (a^m + a^n + 1)$.
\end{cor}

Our method also has implications for orders modulo $n$ when $n$ is composite. For each integer $a$ and each positive integer $n$, the sequence $a, a^2, a^3, \dots$ is eventually periodic modulo $n$. Let $\ell_a(n)$ denote the length of the period. When $a$ and $n$ are coprime, $\ell_a(n)$ is the order of $a$ modulo $n$, while in general, $\ell_a(n)$ is the order of $a$ modulo $n'$, where $n'$ is the largest positive divisor of $n$ coprime to $a$. We show the following analogue of Theorem 1$'$.

\begin{thm}\label{thm:composite} Let $\epsilon(x)$ be a positive-valued function of $x\ge 1$ that tends to $0$ as $x\to\infty$. Let $a$ and $b$ be nonzero integers that are multiplicatively independent. Then for almost all natural numbers $n$ {\rm (}meaning all but $o(x)$ integers $n\le x$, as $x\to\infty${\rm )},
\begin{equation*} \max\{\ell_a(n), \ell_b(n), \ell_{ab}(n), \ell_{a^2 b}(n), \ell_{ab^2}(n)\} > n^{\frac{8}{15}+\epsilon(n)}. \end{equation*}
\end{thm}

The obvious analogue of Theorem \ref{thm:general} also holds; see
 the remarks following our proof of Theorem \ref{thm:composite} for a somewhat more precise result.

\subsection*{Notation} Throughout, the letters $p$ and $q$ are reserved for primes.

\section{Elements of order $\gtrapprox p^{8/15}$: Proof of Theorem \ref{thm:basic}}

For a prime $p$ and integers $a_1,\dots,a_k$ not divisible by $p$, we let $\ell_{a_1,\dots,a_k}(p)$ denote the order of the subgroup of $\F_p^{\times}$ generated by $a_1,\dots,a_k$. The following lemma is a special case of a result of Matthews \cite{matthews82} (alluded to in the introduction), and also of Lemma 1 of Murty and Srinivasan's paper \cite{MS87}.

\begin{lem}\label{lem:matthews} Let $a_1, \dots, a_k$ be nonzero, multiplicatively independent integers. There is a positive constant $C = C_{a_1,\dots,a_k}$ such that, for every $y \ge 1$, the number of primes $p$ with  $\ell_{a_1,\dots,a_k}(p) \le y$ does not exceed $C y^{1+1/k}$.
\end{lem}

Recall that a number $n$ is said to be \textsf{$z$-smooth} (or \textsf{$z$-friable}) if none of its prime factors exceed $z$. The next estimate appears as Theorem 07 on p.\ 4 of Hall and Tenenbaum's monograph \cite{HT88}.

\begin{lem}\label{lem:HT} For certain positive constants $c_1, c_2$ and all real numbers $x\ge y\ge z \ge 2$,  the number of $n\le x$ having a $z$-smooth divisor of size least $y$ is at most
\[ c_1 x \exp(-c_2 \log{y}/\log{z}). \]
\end{lem}

As our final piece of preparation, we recall an elementary fact from group theory that plays a key role in the argument.
\begin{lem}\label{lem:grouplem} Let $G$ be a finite abelian group, and let $g,h \in G$. If $q$ is a prime dividing the order of $g$ but not the order of $h$, then $q$ divides the order of $gh$.
\end{lem}
\begin{proof}[Proof (sketch)] If $g$ and $h$ are commuting elements of a group, with respective orders $m$ and $n$, it is well-known that the order of $gh$ is a multiple of $mn/\gcd(m,n)^2$.
\end{proof}

\begin{proof}[Proof of Theorem \ref{thm:basic}] We will show that for all large $x$, all but $o(x/\log{x})$ primes $p \le x$ are such that one of $a, b, ab, a^2 b$, or $ab^2$ has order exceeding $x^{8/15}/\exp(2\sqrt{\log x})$. For notational simplicity, set
\[ \xi = \log\log{x}. \]

At the cost of discarding $o(x/\log{x})$ `bad' primes $p\le x$, we can assume all of the following conditions hold:
\begin{enumerate}
\item the $\xi$-smooth part (largest $\xi$-smooth divisor) of $p-1$  is at most $\exp(\sqrt{\log x})$,
\item the $\xi$-rough part (the cofactor of the $\xi$-smooth part) of $p-1$  is squarefree,
\item $\ell_a(p), \ell_b(p) > x^{1/2}/\log{x}$, and $\ell_{a,b}(p) > x^{2/3}/\log{x}$,
\end{enumerate}
Indeed, Lemma \ref{lem:HT} implies that the total number of $m\le x$ with $\xi$-smooth part exceeding $\exp(\sqrt{\log x})$ is $o(x/\log{x})$. So the same upper bound certainly holds for the number of these $m$ of the form $p-1$, which handles condition (i). Keeping in mind the Brun--Titchmarsh inequality, we see that the number of primes excluded by (ii) is at most
\begin{align*} \sum_{q > \xi} \pi(x;q^2,1) &\le \sum_{\xi < q \le x^{1/4}} \pi(x;q^2,1) + \sum_{q > x^{1/4}} \frac{x}{q^2} \\
&\ll \frac{x}{\log{x}}\sum_{\xi < q \le x^{1/4}} \frac{1}{q^2} + x^{3/4} \ll \frac{x}{\xi \log{x}};   \end{align*}
this is also $o(x/\log{x})$. That (iii) excludes $o(x/\log{x})$ primes is immediate from Lemma \ref{lem:matthews}.

In the remainder of the proof, we study the prime factorization of the product $$\ell_a(p) \ell_b(p) \ell_{ab}(p) \ell_{a^2 b}(p) \ell_{ab^2}(p).$$

Let $q$ be a prime dividing $p-1$ with $q > \xi$. Then $q\parallel p-1$ (by condition (ii) above). Since the unit group mod $p$ is cyclic, it follows that for each integer $g$ coprime to $p$, the order $\ell_g(p)$ is divisible by $q$ precisely when $g$ is not a $q$th power mod $p$.

Let us use this observation to show that every prime $q>\xi$ dividing $\ell_a(p) \ell_b(p)$ divides at least four of the numbers $\ell_a(p)$, $\ell_b(p)$, $\ell_{ab}(p)$, $\ell_{ab^2}(p)$, $\ell_{a^2b}(p)$. If $q$ divides exactly one of $\ell_a(p)$ and $\ell_b(p)$, then $q$ divides all of $\ell_{ab}(p)$, $\ell_{ab^2}(p)$, $\ell_{a^2b}(p)$ by Lemma \ref{lem:grouplem}. (We use here that $q>2$, which is guaranteed since $x$ is large and $q > \xi$.) Thus the claim holds in this case. So suppose that $q$ divides both $\ell_a(p)$ and $\ell_b(p)$. In that case, $q$ must divide at least two of $\ell_{ab}(p)$, $\ell_{ab^2}(p)$, $\ell_{a^2b}(p)$. Otherwise, at least two of $ab, ab^2$, and $a^2 b$ are $q$th powers mod $p$. But then $a,b$ themselves are $q$th powers mod $p$, contradicting that $q\mid \ell_a(p)$ and $q\mid \ell_b(p)$. (We use here that $q>3$.) So the claim holds in this  case also.

Comparing prime factorizations, we deduce that \[ (\ell_a(p) \ell_b(p) \ell_{ab}(p) \ell_{a^2 b}(p) \ell_{ab^2}(p))^{1/4} \ge \prod_{\substack{q > \xi \\ q\mid \ell_a(p) \ell_b(p)}} q.  \]
Since the $\xi$-rough part of $p-1$ is squarefree, and $\ell_a(p), \ell_p(b)$ divide $p-1$, the right-hand product is the $\xi$-rough part of $\mathrm{lcm}[\ell_a(p),\ell_b(p)] = \ell_{a,b}(p)$, which by (i) and (iii) has size at least
\[ \ell_{a,b}(p)/\exp(\sqrt{\log x}) > x^{2/3}/\exp(2\sqrt{\log x}). \]
This, along with the preceding display, implies that the geometric mean of $\ell_a(p)$, $\ell_b(p)$, $\ell_{ab}(p)$, $\ell_{a^2 b}(p)$ and $\ell_{ab^2}(p)$ is at least $x^{8/15}/\exp(\frac{8}{5}\sqrt{\log x})$. The theorem follows.
\end{proof}

\begin{proof}[Proof of Theorem 1.2] The proof is very similar to that of Theorem \ref{thm:basic}. We keep the notation $\xi=\log\log{x}$. We impose conditions (i) and (ii) from the proof of Theorem \ref{thm:basic}, but replace (iii) with
\begin{enumerate}
\item[(iii$'$)] $\ell_{a_i}(p) > x^{1/2}/\log{x}$ for all $i$, and $\ell_{a_1,\dots,a_k}(p) > x^{k/(k+1)}/\log{x}$.
\end{enumerate}
Conditions (i), (ii), and (iii$'$)  exclude only $o(x/\log{x})$ primes $p\le x$, as $x\to\infty$.

Let $p\le x$ be one of the surviving primes, and let $q>\xi$ be a prime dividing $\ell_{a_1}(p) \cdots \ell_{a_k}(p)$. We claim that $q$ divides $\ell_{a}(p)$ for all but most $N^{k-1}-1$ elements $a \in \Aa$. To see this, let $g$ be a primitive root mod $p$, and write $\log_g(\cdot)$ for the discrete logarithm mod $p$ to the base $g$, which is well-defined (on inputs prime to $p$) as an integer modulo $p-1$. In order for $q$ not to divide $\ell_a(p)$, it must be that $q$ divides $\log_g(a)$. So if we write $a=a_1^{e_1}\cdots a_k^{e_k}$, the number of $a \in \Aa$ for which $q$ does not divide $\ell_a(p)$ is bounded above by the number of $e_1,\dots, e_k$ satisfying the congruence
\begin{equation}\label{eq:boxcongruence} e_1 \log_g(a_1) + \dots + e_k \log_g(a_k) \equiv 0 \pmod{q}, \end{equation}
where $0 \le e_i < N$ for all $i$ and not all $e_i=0$.

By assumption, $q\mid \ell_{a_i}(p)$ for some $i$, and so $q\nmid \log_{g}{(a_i)}$. Relabeling, we can suppose $i=1$. Thus, if we pick $e_j$ arbitrarily for $j=2,\dots,k$, then \eqref{eq:boxcongruence} determines $e_1$ mod $q$. We are assuming $x$ is large, and so $q > \xi > N$. It follows that for any choice of $e_2,\dots,e_k$, there is at most one $e_1 \in [0,N)$ satisfying \eqref{eq:boxcongruence}. Hence, the number of solutions to \eqref{eq:boxcongruence} satisfying our constrains on the $e_i$ is at most $N^{k-1}-1$. Here the `$-1$' comes from noticing that the choice $e_2=e_3=\dots=e_k=0$ forces $e_1=0$, while $e_1=e_2=\dots=e_k=0$ is not allowed.

Since $|\Aa| = N^k-1$, we conclude that each prime $q>\xi$ dividing $\ell_{a_1}(p) \cdots \ell_{a_k}(p)$ divides $\ell_{a}(p)$ for at least $|\Aa| - (N^{k-1}-1) = N^k-N^{k-1}$ choices of $a\in \Aa$. Now arguing as in the proof of Theorem \ref{thm:basic}, we find that as $x\to\infty$,
\begin{align*} \left(\prod_{a \in\Aa} \ell_a(p)\right)^{\frac{1}{N^k-N^{k-1}}} &\ge \prod_{\substack{q > \xi\\ q\mid \ell_{a_1}(p) \cdots \ell_{a_k}(p)}} q \ge \ell_{a_1,\dots,a_k}(p)/\exp(\sqrt{\log x})
\ge x^{1-\frac{1}{k+1} + o(1)},\end{align*}
so that
\[ \left(\prod_{a \in\Aa} \ell_a(p)\right)^{1/|\Aa|} \ge x^{\delta' + o(1)},  \]
where
\begin{equation}\label{eq:deltaprimedef} \delta':= \left(1-\frac{1}{k+1}\right) \frac{N^k-N^{k-1}}{|\Aa|}. \end{equation}
Since $\delta' > \left(1-\frac{1}{k+1}\right) \frac{N^k-N^{k-1}}{N^k} = \left(1-\frac{1}{k+1}\right) \left(1-\frac{1}{N}\right) =\delta$, the exponent on $x$ in the last display exceeds $\delta$ once $x$ is sufficiently large. The theorem follows.
\end{proof}

\section{Primes dividing $\prod_{a\mid A}(a^m+a^n+1)$: Proof of Corollary \ref{cor:skalbavariant}}

The following lemma is due to Ska\l{}ba \cite{skalba04}.

\begin{lem}\label{lem:skalba} let $p$ be a prime, and let $a$ be an integer not divisible by $p$. Suppose that $\ell_a(p) > p^{3/4}$. Then $p$ divides $a^m + a^n+1$ for some positive integers $m,n$.
\end{lem}
\begin{proof} We include the proof for completeness. Note that the hypothesis $\ell_a(p) > p^{3/4}$ implies that $p$ is odd. Let $d = (p-1)/\ell_a(p)$, so that $d < p^{1/4}$. If $d=1$, then $g$ generates $\F_p^{\times}$, and the result is easy: We choose $m$ with $a^m = -1/2$ in $\F_p$ and then take $n=m$. So suppose that $d>1$. The subgroup of ${\F_p}^{\times}$ generated by $a$ coincides with the collection of nonzero $d$th powers in $\F_p$. By Theorem 5 on p.\, 103 of \cite{IR90}, the number of solutions $(x,y)$ to $x^d+y^d=-1$ over $\F_p$ is at least $p-(d-1)^2\sqrt{p}$, and so the number of solutions with $x,y\ne 0$ is at least
\[ p-(d-1)^2\sqrt{p} - 2d > p - d^2 \sqrt{p} > 0. \qedhere\]
\end{proof}

\begin{proof}[Proof of Corollary \ref{cor:skalbavariant}] We keep the notation from the statement and proof of Theorem \ref{thm:general}. We follow the proof of that theorem with $k=5$, with $a_1=2, a_2 = 3, a_3 = 5, a_4 = 7, a_5=11$, and  with $N=10$. That argument shows that for each fixed $\epsilon > 0$, and almost all primes $p$, there is a divisor $a>1$ of $A = (2\cdot3 \cdot 5 \cdot 7\cdot 11)^{9}$ such that
\[ \ell_a(p) \ge p^{\delta'- \epsilon}. \]
Since $\delta' > \delta$, and $\delta = (1-\frac{1}{6})(1-\frac{1}{10})= \frac{3}{4}$, we may complete the proof by fixing $\epsilon$ sufficiently small and applying Lemma \ref{lem:skalba}.
\end{proof}

\begin{rmk} In contrast with Corollary \ref{cor:skalbavariant}, for every positive integer $A$, there is a positive density set of primes $p$ not dividing $\prod_{a \mid A}{(a^n+1)}$ for any integer $n$. Indeed, using quadratic reciprocity one can construct a coprime residue class such that for each  $p$ lying in this class, every prime dividing $A$ is a square mod $p$ but $-1$ is not a square mod $p$. Then the exponential congruence $a^n \equiv -1\pmod{p}$ is not solvable for any $a \mid A$.
\end{rmk}

\section{Composite integers: Proof of Theorem \ref{thm:composite}}

Our argument is modeled on Kurlberg and Pomerance's proof of Theorem 1(a) in \cite{KP05}, which asserts that for each fixed $a\notin \{0,\pm 1\}$, we have $\ell_a(n) > n^{1/2+\epsilon(n)}$ for almost all $n$.

We make crucial use of the following estimate of Kurlberg and Rudnick \cite[\S5]{KR01} (see \cite[Lemma 5]{KP05} for a shorter proof), bounding $\ell_{a}(n)$ from below in terms of the numbers $\ell_a(p)$ for primes $p$ dividing $n$. Here $\lambda(n)$ is \textsf{Carmichael's function}, i.e., the exponent of the multiplicative group mod $n$.

\begin{prop}\label{prop:KR} For each nonzero integer $a$ and each positive integer $n$,
\[ \ell_a(n) \ge \frac{\lambda(n)}{n}\prod_{p\mid n,~p\nmid a} \ell_a(p). \]
\end{prop}

\noindent In our application of Proposition \ref{prop:KR}, we will replace $\frac{\lambda(n)}{n}$ with the lower bound from the  next result, which appears as Lemma 10 of \cite{KP05}.

\begin{lem}\label{lem:lambda} For all large $x$, all but $O(x/(\log{x})^{10})$ values of $n\le x$ satisfy
\[ \lambda(n) > n \exp(-(\log\log{n})^3).\]
\end{lem}

We also use the following result of Ford concerning the number of shifted primes $p-1$ with a divisor from a given interval (see Theorems 1(v) and 6 in \cite{ford08}).

\begin{prop}\label{prop:ford} Suppose $x, y \ge 10^5$ with $y \le \sqrt{x}$ and $2y\le z \le y^2$. Write $z=y^{1+u}$. The proportion of primes not exceeding $x$ for which $p-1$ has a divisor from the interval $(y,z]$ is
\[ \ll u^{\eta} \left(\log\frac{2}{u}\right)^{-3/2},\]
where
\[ \eta := 1-\frac{1+\log\log{2}}{\log{2}} \quad (\approx 0.086). \]
\end{prop}

We now embark on the proof proper of Theorem \ref{thm:composite}.
Replacing the function $\epsilon(t)$ with $\max_{n\ge t} \epsilon(n)$, we can assume that $\epsilon(t)$ is (weakly) decreasing for all $t\ge 2$. Then replacing $\epsilon(t)$ with $\max\{\epsilon(t), 1/\log\log\log{(100 t)}\}$, we can further assume that
\begin{equation}\label{eq:epsilonlower}
 \epsilon(t) \ge 1/\log\log\log(100t)\qquad\text{for all $t\ge 1$}.
\end{equation}

Throughout this proof, we view $a$, $b$, and the function $\epsilon(t)$ as fixed. In particular, when we speak of ``large'' parameters, their required size may depend on $a$, $b$, and $\epsilon(t)$.

Now let $x$ be a large real number, and let $\xi=\log\log{x}$ as before. For primes $p\le x$, we introduce conditions (i)--(iii) defined as follows:
\begin{enumerate}
\item[(i)] the $\xi$-smooth part of $p-1$ is at most $\exp(\sqrt{\log{p}})$,
\item[(ii)] the $\xi$-rough part of $p-1$ is squarefree,
\item[(iii)] $\ell_a(p), \ell_b(p) > p^{1/2}/\log{p}$ and $\ell_{a,b}(p) > p^{2/3}/\log{p}$.
\end{enumerate}
(These are slight variants of conditions (i)--(iii) appearing in the proof of Theorem \ref{thm:basic}.) We  partition the $p\le x$ into classes $U,V,W$, where
\[ U = \{p\le x: p \mid ab\text{ or } p \le \xi \text{ or at least one of (i)--(iii) fails}\}, \]
\[ V = \{p \le x: p \notin U \text{ and } p^{2/3}/\log{p} < \ell_{a,b}(p) \le p^{2/3+5\epsilon(p)}\}, \]
\[ W = \{p \le x: p\notin U \text{ and }\ell_{a,b}(p) > p^{2/3+5\epsilon(p)}\}. \]
We write $\pi_U(t), \pi_V(t)$, and $\pi_W(t)$ for the counts of primes $p\le t$ in the sets $U,V$, and $W$, respectively.

\begin{lem}\label{lem:piUbound} For all large $x$,
\[ \pi_U(t) \ll \frac{\log{x}}{\log\log{x}} + \frac{t}{(\log{t})^{3/2}} + \frac{t}{\log{t} \cdot \log\log{x}} \]
uniformly for $2\le t \le x$.
\end{lem}
\begin{proof} Let $t_0:=\log{x}$. If $t \le t_0$, then $\pi_U(t) \le \pi(t_0) \ll \log{x}/\log\log{x}$, and the lemma  holds. So we will assume that $t > t_0$. Now let $p$ be a prime in $U$ with $p > t_0$. Since $x$ is large, $p > t_0 > |ab|$, and so $p\nmid ab$. Also, $p>t_0 > \xi$. So it must be that one of (i)—(iii) fails.

Let count how many $p \in (t_0,t]$ are such that (i) fails. Since $p > t_0$, the prime $p$ belongs to some interval $(T_j, 2T_j]$, where $T_j = 2^j t_0$ and $j$ is a nonnegative integer. Since (i) fails, the $\xi$-smooth part of $p-1$  exceeds $\exp(\sqrt{\log p})$, which in turn exceeds $\exp(\sqrt{\log T_j})$. But the number of $p\le 2T_j$ for which this occurs is, by Lemma \ref{lem:HT},
\[ \ll 2T_j \exp(-c_2 \sqrt{\log T_j}/\log \xi) \ll T_j \exp(-c_2\sqrt{\log T_j}/\log\log\log x). \]
Since $T_j \ge t_0$, we have $\log\log{T_j} \ge \log\log\log{x}$, and so the last displayed expression is
\[ \ll T_j \exp(-c_2 \sqrt{\log T_j}/\log \log T_j) \ll T_j/(\log T_j)^2. \]
Now we sum on nonnegative integers $j\ge 0$, stopping once the intervals $(T_j,2T_j]$ cover $(t_0,t]$. This gives an upper bound of $O(t/(\log{t})^2)$ on the number of failures of (i).

Condition (ii) is easier to deal with. By Brun--Titchmarsh, the number of $p \in (t_0,t]$ for which (ii) fails is at most
\begin{align*} \sum_{q > \xi} \pi(t;q^2,1) &\le \sum_{\xi < q < t^{1/4}} \pi(t;q^2,1) + \sum_{q > t^{1/4}}\frac{t}{q^2}
\\ &\ll \frac{t}{\log{t}} \sum_{q > \xi} \frac{1}{q^2} + t^{3/4} \ll \frac{t}{\log{t} \cdot \log\log{x}} + t^{3/4}.  \end{align*}

Consider finally the $p \in (t_0,t]$ where (iii) fails. Again, say that $p \in (T_j, 2T_j]$, where $T_j = 2^j t_0$. Then either $\ell_a(p) \le (2T_j)^{1/2}/\log(2T_j)$, $\ell_b(p) \le (2T_j)^{1/2}/\log(2T_j)$, or $\ell_{a,b}(p) \le (2T_j)^{2/3}/\log(2T_j)$. By Lemma \ref{lem:matthews}, the number of $p$ satisfying any of these conditions is $O(T_j/\log(2T_j)^{3/2})$. Summing on $j$ shows that the number of such $p \in (t_0,t]$ is $O(t/(\log{t})^{3/2})$.

Collecting estimates, the number of $p \in (t_0,t]$ belonging to $U$ is
\[ \ll \frac{t}{(\log{t})^{3/2}} + \frac{t}{\log{t}\cdot\log\log{x}}.  \]
Since there  $O(\log{x}/\log\log{x})$ primes not exceeding $t_0$, the lemma follows.
\end{proof}

We now turn to estimating $\pi_V(t)$. Clearly, $\pi_V(t)=0$ if $t \le \xi$.

\begin{lem}\label{lem:piVbound} For all large $x$, and uniformly for $\xi < t\le x$,
\[ \pi_V(t) \ll \epsilon(t')^{1/12} \cdot \frac{t}{\log{t}}, \]
where $t'=t/\log{t}$.
\end{lem}

\begin{proof} Suppose that $p \in V$ with $t/\log{t} < p \le t$.

If $p^{2/3}/\log{p} < \ell_{a,b}(p) \le p^{2/3}$, then $d:=\frac{p-1}{\ell_{a,b}(p)}$ is a divisor of $p-1$ with
\[ t^{1/3}/\log{t} < \frac{1}{2}p^{1/3} \le d < p^{1/3}\log{p} \le t^{1/3}\log{t}. \]
By Proposition \ref{prop:ford}, the number of primes $p \le t$ for which $p-1$ has a divisor in the interval $(t^{1/3}/\log{t},t^{1/3}\log{t}]$ is $O(\pi(t) (\log{t})^{-\eta})$.

On the other hand, if
$p^{2/3} < \ell_{a,b}(p) \le p^{2/3+5\epsilon(p)}$, then $d:=\frac{p-1}{\ell_{a,b}(p)}$ satisfies
\[ \frac{1}{2} t'^{1/3 - 5\epsilon(t')} < \frac{1}{2} p^{1/3-5\epsilon(p)} \le d < p^{1/3} \le t^{1/3}, \]
where $t':=t/\log{t}$. We now apply Proposition \ref{prop:ford} with $y = \frac{1}{2} t'^{1/3 - 5\epsilon(t')}$ and $z=t^{1/3}$. A short calculation, keeping in mind \eqref{eq:epsilonlower}, shows that $z=y^{1+u}$ with $u \le 16\epsilon(t')$. Hence, the number of $p\le t$ for which $p-1$ has a divisor in $(y,z]$ is $O(\pi(t) \epsilon(t')^{\eta})$.

We conclude that the number of $p \le t$ belonging to $V$ is
\[ \ll \pi(t/\log{t}) + \pi(t) (\log{t})^{-\eta} + \pi(t) \epsilon(t')^{\eta}. \]
By \eqref{eq:epsilonlower}, the final summand dominates. The lemma follows upon noting that $\eta > \frac{1}{12}$.
\end{proof}

For each natural number $n\le x$, we let $n_U, n_V, n_W$ be the largest divisors of $n$ supported on the primes in $U, V, W$, respectively. Thus, $n=n_U n_V n_W$.

At the cost of excluding $o(x)$ values of $n\le x$, we can assume that $n/n_U$ is squarefree. Indeed, since $U$ contains all primes up to $\xi$, if this condition fails then $n$ is divisible by $p^2$ for some $p > \xi$, and the number of such $n$ is $O(x\sum_{p>\xi} p^{-2}) = O(x/\xi)$, which is $o(x)$.

Next, we use a first-moment argument to show we can assume, with $o(x)$ exceptions, \[ n_U \le \exp(\log{x}/(\log\log{x})^{1/2}).\]
With $\Lambda(d)$ the von Mangoldt function,
\[ \sum_{n \le x} \log n_U = \sum_{n \le x} \sum_{d \mid n_U} \Lambda(d) = \sum_{d=p^k,~p \in U} \log{p} \sum_{\substack{n \le x \\ d \mid n}}1 \le x \sum_{p \in U} \frac{\log{p}}{p} + O(x). \]
Partial summation, together with the estimate of Lemma \ref{lem:piUbound} for $\pi_U(t)$, shows that
\[ \sum_{p \in U} \frac{\log{p}}{p} \ll \frac{\log{x}}{\log\log{x}}, \]
and thus $\sum_{n\le x} \log n_U \ll x \log{x}/\log\log{x}$. Consequently, the number of $n\le x$ for which $\log n_U > \log{x}/(\log\log{x})^{1/2}$ is $O(x/\sqrt{\log\log{x}})$, which is $o(x)$.

Lemma \ref{lem:piVbound} implies, via partial summation, that $\sum_{p \in V} \frac{\log{p}}{p} = o(\log{x})$, as $x\to\infty$. So by an argument analogous to that in the last paragraph, we have
\[ n_V \le x^{1/10} \]
for all but $o(x)$ values of $n\le x$.

We will also assume that
\[ n \ge x^{1/2}, \]
that
\[ \lambda(n) \ge n \exp(-(\log\log{x})^3), \]
and that
\[ \omega(n) \le 2\log\log{x}.\]
That the first condition excludes only $o(x)$ values of $n\le x$ is clear. That the same holds for the second condition follows from Lemma \ref{lem:lambda}. That the third condition admits only $o(x)$ exceptions is a consequence of as well-known theorem of Hardy and Ramanujan.

Since $n\ge x^{1/2}$, while $n_U n_V \le x^{1/5}$ (say) once $x$ is large, our assumptions imply that
\[ n_W = \frac{n}{n_U n_V} \ge x^{3/10} \]

Suppose that $x$ is large, and that the natural number $n\le x$ is not in any of the exceptional sets indicated so far.  By Proposition \ref{prop:KR},
\[ \ell_a(n)\ell_b(n)\ell_{ab}(n) \ell_{a^2 b}(n) \ell_{ab^2}(n) \ge \exp(-5(\log\log{x})^3) \prod_{p \mid n_V n_W} \ell_a(p)\ell_b(p)\ell_{ab}(p) \ell_{a^2 b}(p) \ell_{ab^2}(p). \]
The argument used in the proof of Theorem \ref{thm:basic}, now using our modified conditions (i)--(iii), shows that for each $p$ dividing $n_V$,
\[ (\ell_a(p)\ell_b(p)\ell_{ab}(p) \ell_{a^2 b}(p) \ell_{ab^2}(p))^{1/4} \ge \ell_{a,b}(p)/\exp(\sqrt{\log{p}}), \]
and thus (for large $x$)
\[ \ell_a(p)\ell_b(p)\ell_{ab}(p) \ell_{a^2 b}(p) \ell_{ab^2}(p) \ge p^{8/3}/\exp(5\sqrt{\log{p}}).  \]
Using the better lower bound for $\ell_{a,b}(p)$ available for $p$ in $W$, the same argument shows that for each $p$ dividing $n_W$,
\[ \ell_a(p)\ell_b(p)\ell_{ab}(p) \ell_{a^2 b}(p) \ell_{ab^2}(p) \ge p^{8/3 + 20\epsilon(p)}/\exp(4\sqrt{\log p}).\]
Substituting back above the results of the last two displays, and using that $\epsilon(p) \ge \epsilon(n)$ for each $p$ dividing $n$, we find that
\begin{multline*} \ell_a(n)\ell_b(n)\ell_{ab}(n) \ell_{a^2 b}(n) \ell_{ab^2}(n) \ge \exp(-5(\log\log{x})^3) \cdot \exp\bigg(-5\sum_{p\mid n_V n_W} \sqrt{\log{p}}\bigg) \\
\times \prod_{p \mid n_V n_W} p^{8/3}  \prod_{p \mid n_W} p^{20\epsilon(n)}.  \end{multline*}
Recall that $n_V n_W= n/n_U$ is squarefree. Hence,
\begin{align*} \prod_{p \mid n_V n_W} p^{8/3} \prod_{p \mid n_W} p^{20\epsilon(n)} &= (n_V n_W)^{8/3} n_W^{20 \epsilon(n)} \\
&= (n^{8/3}/n_U^{8/3}) \cdot {n_W}^{20\epsilon(n)}\\
&\ge n^{8/3} \exp(-3\log{x}/(\log\log{x})^{1/2}) \cdot x^{6\epsilon(n)} \\
&\ge n^{8/3 + 6\epsilon(n)} \exp(-3\log{x}/(\log\log{x})^{1/2}).
 \end{align*}
Moreover,
\[ \exp\bigg(-5\sum_{p \mid n_V n_W} \sqrt{\log{p}}\bigg) \ge \exp(-5\omega(n)\sqrt{\log{x}}) \ge \exp(-10\sqrt{\log x}\log\log{x}). \]

Putting together our estimates, we find that (for large $x$)
\[ \ell_a(n)\ell_b(n)\ell_{ab}(n) \ell_{a^2 b}(n) \ell_{ab^2}(n) \ge n^{8/3+5\epsilon(n)}\cdot (n^{\epsilon(n)}  \exp(-4\log{x}/(\log\log{x})^{1/2})). \]
Using that $n \ge x^{1/2}$ along with the lower bound \eqref{eq:epsilonlower}, we see that the parenthesized right-hand factor is larger than $1$ (for large $x$). So taking fifth roots, the geometric mean of $\ell_a(n)$, $\ell_b(n)$, $\ell_{ab}(n)$,  $\ell_{a^2 b}(n)$, and $\ell_{ab^2}(n)$ exceeds $n^{8/15 + \epsilon(n)}$.

\begin{rmk} An analogous argument will establish the following analogue of Theorem \ref{thm:general}. Let $\epsilon(x)$ be a positive-valued function of $x\ge 1$ that tends to $0$ as $x\to\infty$. Let $a_1,\dots,a_k$ be nonzero integers that are multiplicatively independent. Let $N$ be a positive integer, and let
\[ \Aa = \{a_1^{e_1} a_2^{e_2} \cdots a_k^{e_k}: \text{each }0 \le e_i < N, \text{not all $e_i=0$}\}. \]
For almost all $n$, there is an $a\in \Aa$ with
\[ \ell_a(n) > n^{\delta' + \epsilon(n)}, \]
where $\delta'$ is defined as in \eqref{eq:deltaprimedef}.
\end{rmk}

\bibliographystyle{amsplain}
\bibliography{nearprimitive}

\end{document}